\documentclass[a4paper,11pt,reqno]{amsart}
\usepackage{amsmath,amssymb,amsthm}

\usepackage{hyperref}
\usepackage{color}

\usepackage{latexsym}
\usepackage{epic,latexsym,bezier,amsbsy,color,enumerate,amsfonts,amscd}
\usepackage[latin1]{inputenc}

\newcommand{\CC}{\mathbb{C}}
\newcommand{\NN}{\mathbb{N}}

\newcommand{\TT}{\mathbb{T}}

\newtheorem{theorem}{Theorem}[section]
\newtheorem{corollary}[theorem]{Corollary}
\newtheorem{lemma}[theorem]{Lemma}
\newtheorem{proposition}[theorem]{Proposition}

\theoremstyle{definition}
\newtheorem{definition}[theorem]{Definition}

\theoremstyle{remark}

\newcommand{\Hmm}[1]{\leavevmode{\marginpar{\tiny%
$\hbox to 0mm{\hspace*{-0.5mm}$\leftarrow$\hss}%
\vcenter{\vrule depth 0.1mm height 0.1mm width \the\marginparwidth}%
\hbox to 0mm{\hss$\rightarrow$\hspace*{-0.5mm}}$\\\relax\raggedright
#1}}}

\begin{document}

\title[A characterization of  discrete  spectrum]{An autocorrelation and discrete  spectrum  for dynamical systems on metric spaces}

\author[Daniel Lenz]{Daniel Lenz$^1$}

\address{ $^1$ Mathematisches Institut, Friedrich Schiller Universit\"at Jena,
  D-03477 Jena, Germany, daniel.lenz@uni-jena.de,
  URL: http://www.analysis-lenz.uni-jena.de/ }

\begin{abstract} We study dynamical systems $(X,G,m)$ with a compact
metric space $X$ and a locally compact, $\sigma$-compact,  abelian
group $G$. We show that such a  system has discrete spectrum if and
only if a certain space average  over the metric is a Bohr almost
periodic function. In this way, this average over the metric plays
for general dynamical systems  a similar role as the autocorrelation
measure  plays in the study of aperiodic order for  special
dynamical systems based on point sets.
\end{abstract}

\maketitle


\section*{Introduction}\label{Introduction}
Dynamical systems with  discrete  spectrum have met substantial
interest in the past. In particular, they  have attracted quite some
attention in recent years.

One reason is that such systems play an important role in the
investigation of aperiodic order. Aperiodic order, also known as
mathematical theory of quasicrystals, has emerged as  a substantial
topic of research over the last three decades,  see e.g. \cite{TAO}
for a recent monograph and \cite{KLS} for a recent collection of
surveys. A key feature of aperiodic order is the occurrence of
(pure) point diffraction. Due to a collective effort over the years
pure point diffraction is now understood as discrete spectrum of
suitable associated dynamical systems \cite{Dwo,LMS,Gou,BL,LS,LM}.

Another instance of the recent interest in  discrete spectra can be
found in a series of works which analyse  such spectra via weak
notions of equicontinuity, \cite{GR,GM,GD}. These works provide in
particular a characterization of discrete spectrum  (see
\cite{FGJ,Ver} for related  work as well) and a characterization of
discrete spectrum with continuous eigenfunctions and unique
ergodicity.

The dynamical systems $(X,G,m)$ underlying the investigation of
aperiodic order (and defined in  detail in Section
\ref{Connections}) have a special structure. The compact space $X$,
on which the locally compact, $\sigma$-compact,   abelian group $G$
acts,  consists of point sets or more generally measures.
Accordingly, these systems are known as \textit{translation bounded
measure dynamical systems} (TMDS). The fact that the points of $X$
are measures  allows one to pair elements of $X$ with elements from
the vector space $C_c (G)$ of continuous compactly supported
functions on the group resulting in a map  $N$ from $C_c (G)$ to
functions on $X$. Via this map one can then define the
\textit{autocorrelation measure} $\gamma$ associated to $(X,G,m)$.
The Fourier transform of the autocorrelation measure is known as
\textit{diffraction measure}. The diffraction measure or,
equivalently, the  autocorrelation measure encodes a remarkable
amount of information on the original system. In fact, a main result
of the theory (already mentioned above and discussed in Section
\ref{Connections} in more detail) can be stated as follows:

\medskip

\textbf{Result - TMDS. \cite{LMS,Gou,BL,LS,LM}} \textit{The TMDS
$(X,G,m)$ has discrete spectrum if and only if the measure $\gamma$
is strongly almost periodic. In this case, the group generated by
the frequencies of $\gamma$ is the group of eigenvalues of
$(X,G,m)$.}

\medskip

In a  general dynamical system $(X,G,m)$ (again with notation to be
explained in detail later in Section \ref{Preliminaries}) the points
of $X$ can not be paired with elements of $C_c (G)$. Hence, such a
system does not admit an autocorrelation. However, if $d$ is a
metric on $X$ inducing the topology, then -  as we will show below -
the function
$$\underline{d} : G\longrightarrow [0,\infty), \underline{d}(t) =
\int_X d(x,tx) dm(x),$$ can serve as a convenient analogue to the
autocorrelation. Indeed, our main abstract result  reads as follows:

\medskip

\textbf{Main result.} (Compare Theorem 4.1 below.) \textit{The
dynamical system $(X,G,m)$ has   discrete spectrum if and only if
$\underline{d}$ is almost periodic in the sense of Bohr. In this
case, the group of eigenvalues is generated by the frequencies of
$\underline{d}$.}

\medskip

For general dynamical systems over metric spaces, this result
provides an analogue to the result above for TMDS. As a consequence
we also obtain (in Corollary \ref{cor-sol}) a converse to a result
of \cite{Sol}. This is particular remarkable as it is mentioned in
\cite{Sol} that 'it is unlikely' that such a converse hold. Finally
-  as it is to be expected  - our considerations allow us to also
reprove the above result for TMDS  provided the group $G$ is
metrizable, see Section \ref{Connections}.

Our approach to  discrete spectrum is somewhat complementary to the
approach in the  quoted works \cite{GR,GM,GD,Ver} above. A central
quantity in these works is the pseudometric $\overline{d}$ on $X$,
which arises by averaging $d$ over $G$,
$$\overline{d} (x,y) =\limsup_{n\to \infty} \frac{1}{m_G (F_n)} \int_{F_n}
d(t x, ty) d m_G (t),$$ where $(F_n)$ is a F{\o}lner sequence and
$m_G$ denotes Haar measure on $G$. The discreteness of the spectrum
(and related phenomena) is then encoded in equicontinuity and
covering properties with respect to the topology induced by
$\overline{d}$. In comparison our key quantity $\underline{d}$ can
(essentially) be seen as a metric on $G$, which arises by averaging
over $X$. Discreteness of the spectrum is then encoded in almost
periodicity properties of $\underline{d}$. These two points of view
are certainly related. For example, it is possible to use our result
to reprove parts of the abstract considerations of \cite{GD} on
spectral isomorphy. This and more will be considered elsewhere.

\medskip

The necessary notation and set-up for our investigations in provided
in Section \ref{Preliminaries}. The main technical tools are
gathered in Section \ref{Functions}. There we also introduce the
domination relation $\prec$ which is a key ingredient in our
analysis. With these tools we can then provide various
characterizations of  almost periodicity  of $\underline{d}$ in
Section \ref{Almost}. The main result, given in Section
\ref{Characterization},  is then a rather direct consequence of
these characterizations.

\section{Preliminaries} \label{Preliminaries}
In this section we set up notation and recall a few standard facts
on dynamical systems. The material is well-known.

\bigskip

We consider  a compact space $X$ equipped with a continuous  action
$$G\times X\longrightarrow X, (t,x)\mapsto
tx,$$ of the locally compact, $\sigma$-compact,  abelian group $G$
and a probability measure $m$, which is invariant under the action
of $G$. We then call $(X,G,m)$ a \textit{dynamical system over the
space $X$}. Throughout we will assume that the topology of $X$ is
induced by a metric. This metric will usually be denoted by  $d$.

The action of $G$ on $X$ induces  unitary operators $T_t : L^2
(X,m)\longrightarrow L^2 (X,m)$ with $$ T_t f  = f(t\;  \cdot)$$ for
each $t\in G$. Here, $L^2 (X,m)$ is the Hilbert space of
(equivalence classes of) square integrable functions on $X$. It is
equipped with the inner product
$$\langle f, g\rangle = \int_X \overline{f} g\,  dm$$
and the associated norm
$$\|f\| :=\sqrt{\langle f, f\rangle}$$
for $f,g\in L^2 (X,m)$. We will be particularly interested in
continuous functions on $X$ and will denote the vector space of all
such functions by $C(X)$.

\smallskip

An $f\in L^2 (X,m)$ with $f\neq 0$ is called an
\textit{eigenfunction to the eigenvalue} $\gamma \in\widehat{G}$ if
$T_t f = \gamma (t) f$ holds  for each  $t\in G$. Here,
$\widehat{G}$ is the \textit{dual group of} $G$ consisting of all
continuous group homomorphisms $\gamma : G\longrightarrow \{z\in \CC
: |z| = 1\}$ equipped with multiplication of functions  and complex
conjugation as product and inverse respectively. We denote the group
generated by the set of eigenvalues as the \textit{group of
eigenvalues}. If the dynamical system is minimal (i.e. each orbit is
dense) or ergodic (i.e. any measurable invariant set has measure $0$
or $1$),  then the set of eigenvalues forms already a group.

\smallskip

The dynamical system $(X,G,m)$ is said to have \textit{ discrete
spectrum} if there exists an orthonormal basis of $L^2 (X,m)$
consisting of eigenfunctions.

\smallskip

Whenever $f$ is a bounded function from a set $Y$ to the complex
numbers, we define the \textit{supremum norm} of $f$ as
$\|f\|_\infty := \sup\{|f(y)| : y\in Y\}$. We will be interested in
the cases $Y = G$ and $Y = X$.

\smallskip

We write the group operation on $G$ additively and denote the
neutral element of $G$ by $0$.

\section{The functions $\underline{e}$ and $e'$}\label{Functions}

Let $(X,G,m)$ be a dynamical system over a compact  space $X$. A
\textit{pseudometric} on a set $Y$  is a function  $e: Y\times
Y\longrightarrow [0,\infty)$ satisfying $ e(x,x) =0$, $e(x,y) =
e(y,x)$ and $e(x,y) \leq e(x,z) + e(z,y)$ for all $x,y,z\in Y$. We
will be interested in pseudometrics on $X$ and on  $G$.

To a continuous pseudometric $e$ on $X$ we associate the functions
$$ \underline{e} : G\longrightarrow [0,\infty), \;
\underline{e}(t):=\int_X e(x,tx) dm(x),$$ and
$$e' : G\times G\longrightarrow [0,\infty),\; e' (s,t) := \int_X e
(sx, tx) dm(x).$$

A short computation (using the invariance of $m$)  gives
$$ e' (s,t) = \underline{e}(s-t) \mbox{ and } \underline{e}(s) = e'(0,s).$$
For this reason properties of $\underline{e}$ and of  $e'$ are
strongly connected and it usually suffices to study one of these
functions. A few basic properties of $e'$ are gathered next.

\begin{proposition}[Basic properties of $e'$]
The function $e'$ is a continuous, bounded  and $G$-invariant
pseudometric.
\end{proposition}
\begin{proof} Continuity of $e'$ is clear from continuity of the
group action and compactness of $X$. Boundedness of $e'$ follows as
$m$ is a probability measure. The $G$-invariance is clear from the
formula $e' (s,t) = \underline{e}(s-t)$. It remains to show that
$e'$ is a pseudometric. This follows easily as $e$ is a
pseudometric.
\end{proof}

\begin{lemma}[Functions inducing invariant metrics] \label{inducing}
 Let $F : G\longrightarrow [0,\infty)$ be a function on
$G$ such that $F'(t,s):= F(t-s)$ is a pseudometric. Then, $F$
satisfies $F(0) = 0$, $F(-s) = F(s)$ and $|F(s)- F(t)|\leq F(s-t)$
and $$\|F (s + \cdot) - F(t + \cdot)\|_\infty  = F(s-t)$$ for all
$s,t\in G$. If $F$ is continuous (at $t = 0$)  it is uniformly
continuous on $G$.
\end{lemma}
\begin{proof} As $F'$ is a pseudometric we infer
$F(0) = F' (0,0) = 0 \mbox{ and } F(s) = F' (s,0) = F' (0,s) =
F(-s)$ as well as
$$|F(s) - F(t)| = |F' (s,0) -F'(t,0)| \leq F' (t,s) = F(t-s).$$
for all $s,t\in G$. From this inequality we  directly obtain the
inequality $\|F (s + \cdot) - F(t + \cdot)|_\infty  \leq  F(s-t)$
for all $s,t\in G$.  The  reverse inequality $\geq$ follows by
inserting the value $-t$. Now, the last statement is clear.
\end{proof}

\begin{proposition}[Basic properties of $\underline{e}$]
Let $e$ be a continuous  pseudometric on $X$.
 The function $\underline{e}$  is bounded and satisfies $\underline{e}(0)
= 0$, $\underline{e}(s) = \underline{e}(-s)$ as well as
$|\underline{e} (s) -\underline{e}(t) |\leq \underline{e}(s-t)$ and
$$\|\underline{e}(\cdot + s) -\underline{e}(\cdot + t)\|_\infty =
\underline{e}(s-t)$$ for all $s,t\in G$.  Moreover, $\underline{e}$
is uniformly continuous.
\end{proposition}
\begin{proof} As $X$ is compact and $e$ is continuous, the function
$\underline{e}$ is bounded and continuous. As for the remaining
statements we can apply the previous lemma with  $F = \underline{e}$
and $F' = e'$. \end{proof}

Of course, the functions $\underline{e}$ and $e'$ depend on $e$. So,
one may wonder how they change if $e$ is replaced by another
pseudometric.  To investigate  this we introduce the following
concept.

\begin{definition}[The relation $\prec$] Let $f$ and $g$ be functions from a set $Y$ to the complex
numbers. Then, $f $ is said to \textit{dominate} $g$ written as
$g\prec f$ if for all $\varepsilon >0$ there exists a $\delta>0$
such that $|g(y)|\leq \varepsilon$ whenever $|f(y)|\leq \delta$
holds.  If both $f$ dominates $g$ and $g$ dominates $f$ we say that
$f$ and $g$ are equivalent and  write $f\sim g$.
\end{definition}


The following statement is a rather  direct consequence of uniform
continuity of continuous functions on compact sets.

\begin{proposition} Let $e$ be a pseudometric on  the compact $X$.
Let $d$ be a metric on $X$ inducing the topology.  Then, $e$ is
continuous if and only if it is dominated by $d$.
\end{proposition}

By the previous proposition, two metrics $d$ and $e$ on the compact
$X$ giving the topology are equivalent.

\begin{lemma}\label{thelemma}
Let $(X,G,m)$ be a dynamical system. Let $e_1$ and $e_2$ be
continuous pseudometrics on $X$ with $e_1\prec e_2$. Then,
$$\underline{e_1}\prec \underline{e_2} \mbox{ and }  e_1' \prec  e_2'.$$
\end{lemma}
\begin{proof} It suffices to show the statement for
$\underline{e_1}$ and $\underline{e_2}$.   We have to show that  for
each $\varepsilon
>0$ there exists a $\delta>0$ such that $\underline{e_2} (t)  <\delta$
(for  $t\in G$)  implies $\underline{e_1} (t) <\varepsilon$. So, let
$\varepsilon >0$ be given. Without loss of generality we can assume
$e_1,e_2\leq 1$.

By  $e_1\prec e_2$, there exists  $\varepsilon_1>0$ with
$$e_2(x,y)\leq \varepsilon_1\Longrightarrow e_1(x,y) \leq
\frac{\varepsilon}{2}.$$ Without loss of generality we can assume
$$ \varepsilon_1 \leq \frac{\varepsilon}{2}. $$
Set $\delta:=\varepsilon_1^2$ and consider a $t\in G$ with
$\underline{e_2} (t) \leq \delta$. Setting
$$M:=\{ x : e_2(x,tx)\geq \varepsilon_1\}$$
we then find
$$\varepsilon_1 m(M)\leq \int_X e_2 (x,tx) dm(x) = \underline{e_2}(t) <
\varepsilon_1^2.$$ This gives
$$m(M) \leq \varepsilon_1.$$
By construction we have $e_2 (x,tx)< \varepsilon_1$ and hence $e_1
(x,tx) \leq \frac{\varepsilon}{2}$ for $x\in X\setminus M$. Given
this a short computation shows
\begin{eqnarray*}
\underline{e_1}(t) &=& \int_X e_1 (x,tx) dm(x)\\
&=& \int_M e_1 (x,tx) dm(x) + \int_{X\setminus M} e_1 (x,tx) dm(x)\\
&\leq & \|e_1\|_\infty m(M) + \frac{\varepsilon}{2} m(X\setminus M)\\
&\leq & \varepsilon_1 + \frac{\varepsilon}{2}\\
&\leq & \varepsilon.
\end{eqnarray*}
This finishes the proof.
\end{proof}

For us certain (pseudo)metrics will be of special interest: To each
continuous $f : X\longrightarrow \CC$ we associate the pseudometric
$$e_f : X\times X\longrightarrow [0,\infty), e_f (x,y) := |f(x) -
f(y)|.$$ As $f$ is continuous, so is $e_f$. In particular, $e_f$ is
dominated by the metric $d$.  Let now for $n\in\NN$ functions   $f_n
\in C(X)$ and $c_n \geq 0$ with $\sum c_n \|f_n\|_\infty < \infty$
be given. Then,
$$e_{(f_n), (c_n)} :=\sum_n c_n e_{f_n}$$
is a continuous pseudometric. To an $f \in C(X)$   we can also
associate the function
$$F_f : G\longrightarrow [0,\infty), F_f
 (t):=\|f - T_t f\|.$$
To  $f_n\in C(X) $ and $c_n \geq 0$ with $\sum c_n \|f\|_\infty$ we
can moreover associate the function
 $$F_{(f_n), (c_n)} :=\sum_n c_n F_{f_n}.$$

\begin{proposition}\label{thecorollary}

(a) Let $f$ be a continuous function on $X$. Then,  $\underline{e}_f
\prec \underline{d}$.

(b) Let $f$ be a continuous function on $X$. Then, $ F_f  \leq
\sqrt{2 \|f\|_\infty} \cdot  \sqrt{\underline{e}_f}$ and
 $ \underline{e}_{f} \leq F_f$ hold. In particular, $F_f\sim
  \underline{e}_f$ and $F_f\prec \underline{d}$.

(c) Let continuous functions  $f_n$ on $X$ and $c_n >  0$, $n\in
\NN$, with $\sum_n c_n \|f_n\| <\infty$ be given such that the
$(f_n)$ separate the points of $X$. Then,
$$\underline{e}_{(f_n), (c_n)} \sim F_{(f_n), (c_n)} \sim \underline{d}.$$
\end{proposition}
\begin{proof}
(a)  We have already discussed that $e_f$ is continuous. Hence, it
is dominated by $d$. Thus, the previous  lemma gives
$\underline{e}_f \prec \underline{d}$.

\smallskip

(b) To show the bound on $F_f$ we compute
\begin{eqnarray*}
F_f (t) &=&\|f -  T_t f\| \\
&=&\left(\int |f(x) - f(tx)|^2 dm(x)\right)^{1/2}\\
&\leq & \sqrt{2  \|f\|_\infty} \left(\int |f(x) - f(tx)| dm(x)\right)^{1/2}\\
&=&\sqrt{2  \|f\|_\infty} \cdot \sqrt{\underline{e}_{f} (t)}.
\end{eqnarray*}

To show the bound on $\underline{e}_f$,  we note that  $m$ is a
probability measure, and hence  Cauchy-Schwarz inequality gives
$\int_X |g|dm \leq \|g\|$ for all $g$ in $L^2 (X,m)$. Thus, we can
estimate
\begin{eqnarray*}
\underline{e}_f (t) &=& \int_X e_f (x,tx)dm (x)\\
&=& \int |f(x) - f(tx)| d m (x)\\
&\leq &\|f -  T_t f\|\\
&=& F_f(t)
\end{eqnarray*}
for each $t\in G$.

The preceding two  bounds give $F_f \sim \underline{e}_f$. Invoking
(a), we then also infer $ F_f \prec \underline{d}$.

\smallskip

(c) From (b) and the summability  condition on the sequence $(c_n)$
we directly infer $\underline{e}_{(f_n), (c_n)}\leq F_{(f_n),
(c_n)}$ as well as
$$F_{(f_n), (c_n)} \leq \sum_n c_n \sqrt{ 2\|f_n\|_\infty}
\sqrt{\underline{e}_{f_n}} \leq \left(2\sum_n c_n
\|f_n\|_\infty\right)^{1/2} \sqrt{ \underline{e}_{(f_n), (c_n)}}.$$
This gives
$$\underline{e}_{(f_n), (c_n)}\sim F_{(f_n), (c_n)}.$$

\smallskip

 From (a) and the summability condition on the sequence $(c_n)$  we also easily find
$\underline{e}_{(f_n), (c_n)}  \prec \underline{d}$. It remains to
show $\underline{d} \prec\underline{e}_{(f_n), (c_n)}$. Now, by the
assumptions the function $e:= e_{(f_n), (c_n)}$  is a continuous
metric. As $X$ is a compact Hausdorff space this metric must then
generate the topology as well. Thus, $d$ is continuous with respect
to the metric $e$ and vice versa. Thus, $d$ and $e$ are equivalent.
Thus, the previous lemma  gives $\underline{d} \sim \underline{e}$.
 This finishes the proof.
\end{proof}

\section{Almost periodicity and $\underline{d}$}\label{Almost}

\bigskip

Recall that a subset $\mathcal{T}$ of $G$ is called
\textit{relatively dense} if  there exists a compact $K\subset G$
with Minkowski sum
$$\mathcal{T} + K :=\{ \tau + k  : \tau\in \mathcal{T}, k\in K\}$$
equal to $G$.  A uniformly continuous function $F : G\longrightarrow
\CC$ is called \textit{almost periodic (in the sense of  Bohr)} if
for any $\varepsilon
>0$  the set
$$ \{ t\in G : \|F(t + \cdot) - F\|_\infty \leq \varepsilon\}$$
is relatively dense. This is the case if and only if the
\textit{hull of $F$} defined by
$$\TT (F):=\overline{\{ F(t+ \cdot) : t\in G\}}^{\|\cdot\|_\infty}$$
is compact. In this case  the hull has a unique group structure
making it into  a topological group such that
$$j: G\longrightarrow \TT (F), \;  t\mapsto F(t + \cdot),$$
is a continuous group homomorphism. Clearly, the map $j$ has dense
range. Hence, the group $\TT (F)$ must be abelian. Moreover,  the
dual map
$$\widehat{\TT (F)}\longrightarrow \widehat{G}, \gamma \mapsto
\gamma \circ j,$$ is injective. Hence, any element $\gamma\in
\widehat{\TT(F)}$ can be considered as an element of $\widehat{G}$
and this is how we will think about elements of $\widehat{\TT(F)}$.

As $\TT(F)$ is a compact group it carries a unique normalized
invariant measure $m_{\TT(F)}$  and the elements of $\widehat{\TT
(F)}$ form an orthonormal basis in the Hilbert space $L^2 (\TT(F),
m_{\TT(F)})$. Consequently any continuous function $h$ on $\TT(F)$
can then be expanded uniquely in a \textit{Fourier series}
$$ h = \sum_{\gamma\in \widehat{\TT(F)}} c_\gamma^h  \gamma, $$
where the sum converges in the $L^2$-sense. Whenever $H$ is a
continuous  function on $G$ such that there exists a (necessarily
unique) continuous $h$ on $\TT(F)$ with $H = h \circ j$ we call
$$\{\gamma \circ j : c_\gamma^h\neq 0\}\subset \widehat{G}$$
the \textit{set of frequencies of $H$}. Clearly, one such function
is given by  $H =  \delta\circ j$ with the function $\delta : \TT
(F)\longrightarrow \CC, \delta (E) = E(0).$ In this case the group
generated by the frequencies of $H$ is just  $\widehat{\TT (F)}$.

\smallskip

We will be interested in proving almost periodicity of functions
such as $\underline{e}$ for a continuous pseudometric $e$ on $X$,
$F_f$ for $f\in C(X)$ and $F_{(c_n), (f_n)}$ for $f_n \in C(X)$ and
$c_n\geq 0$ with $\sum_n c_n \|f_n\| <\infty$.  All of these  are
functions  $F :G\longrightarrow [0,\infty)$ such that $F'$ with $F'
(s,t):= F(s-t)$  is a continuous pseudometric on $G$. Hence, they
are  continuous $F: G\longrightarrow [0,\infty)$  with $F(0) =0$ and
\begin{eqnarray}\label{condition}
\|F (s + \cdot) - F (t+ \cdot)\|_\infty = F (s-t)
\end{eqnarray} for
all $s,t \in G$ (compare Lemma \ref{inducing}). For such functions
we clearly have the following criterion for almost periodicity.

\begin{lemma} \label{observation}
A continuous  $F: G\longrightarrow [0,\infty)$  with $F(0) =0$ and $
\|F (s + \cdot) - F_j(t+ \cdot)\|_\infty = F (s-t)$ for all $s,t \in
G$ is almost periodic if  and only if for every $\varepsilon
>0$  the set
\begin{eqnarray}\label{settt}\{ \tau \in G : |F(\tau)|\leq \varepsilon \}
\end{eqnarray}
 is relatively dense.
\end{lemma}
Accordingly, proving relative denseness of sets as in \eqref{settt}
 will be our main tool in dealing with almost-periodicity.

As relative denseness of sets as in \eqref{settt} is clearly
preserved under $\prec$ we obtain the following consequence.

\begin{proposition} Let $F_j: G\longrightarrow [0,\infty)$, $j\in\{1,2\}$,  be
continuous   with $F_j (0) =0$ and
$$ \|F_j(s + \cdot) - F_j(t+ \cdot)\|_\infty = F_j(s-t)$$
for all $s,t \in G$ and   $F_1 \prec F_2$. Then, $F_1$ is almost
periodic if $F_2$ is almost periodic.
\end{proposition}
\begin{proof} By $F_1 \prec F_2$ relative denseness of
$\{ \tau \in G : |F_2(\tau)|\leq \varepsilon \}$ for each
$\varepsilon >0$ implies relative denseness of $\{ \tau \in G :
|F_1(\tau)|\leq \delta \}$ for each $\delta >0$. Now, the  proof
follows from Lemma \ref{observation}.
\end{proof}

The following lemma gathers various equivalent versions of almost
periodicity of $\underline{d}$. Our main result will be a rather
direct consequence of this lemma. The equivalence between (v), (v')
and (vi) is well-known. As we will need this equivalence, we include
a discussion for completeness reasons.

\begin{lemma}\label{characterization-ap} Let $(X,G,m)$ be a dynamical system.
Then, the  following assertions are equivalent:

\begin{itemize}

\item[(i)] The function $\underline{d}$ is almost periodic.

\item[(ii)] For each $f\in C(X)$ the function $\underline{e}_f$
is almost periodic.

\item[(ii)'] For each  $f\in C(X)$ the function $F_f$ is almost
periodic.

\item[(iii)] The set of $f\in C(X)$ for which
$\underline{e}_f$ is almost periodic separates the points of $X$.

\item[(iii)'] The  set of  $f\in C(X)$ for which
$F_f$ is almost periodic separates the points of $X$.

\item[(iv)] The function $\underline{e}_{(f_n), (c_n)}$ is almost periodic for
one (each) set of  functions  $f_n$ in $C(X)$ and $c_n > 0$, $n\in
\NN$, with $\sum_n c_n \|f_n\| <\infty$  such that the $(f_n)$
separate the points of $X$.

\item[(iv)'] The function $F_{(f_n), (c_n)}$ is almost periodic.
for one (each) set of functions  $f_n$ in $C(X)$  and $c_n > 0$,
$n\in \NN$, with $\sum_n c_n \|f_n\| <\infty$  such that the $(f_n)$
separate the points of $X$.

\item[(v)] For any $f\in C(X)$  the function $G\longrightarrow
\CC, t\mapsto \langle f, T_t f\rangle$ is almost periodic.

\item[(v)'] For any $f\in L^2 (X,m)$ the function $G\longrightarrow
\CC, t\mapsto \langle f, T_t f\rangle$ is almost periodic.

\item[(vi)] The set of $f\in C(X)$ for which  $G\longrightarrow
\CC, t\mapsto \langle f, T_t f\rangle$ is almost periodic separates
the points of $X$.

\end{itemize}

If one of these equivalent conditions holds the group generated by
the frequencies of $\underline{d}$ is the same as the group
generated by the  frequencies of all  functions of the form
$G\longrightarrow \CC, \;  t\mapsto \langle f, T_t f\rangle$ for
$f\in L^2(X,m)$.
\end{lemma}
\begin{proof} We first discuss the equivalence of the various
statements:  Note that the  functions in question in statements (i)
to (iv)'  are continuous, vanish at the neutral element of $G$  and
satisfy the assumption \eqref{condition}. Thus, we can invoke the
previous proposition to study their almost-periodicity.  Given this,
the equivalence between statements (i) to (iv)' is  a direct
consequence of Proposition \ref{thecorollary}. Note that (iii)
implies (iv) and (iii)' implies (iv)' as each set of continuous
functions that separates the points must contain a countable subset
of functions which also separates the points of $X$ (due to
compactness of $X$).

\smallskip

The implications (v)'$\Longrightarrow$ (v) $\Longrightarrow $ (vi)
are clear.  The crucial ingredient for the remaining part of the
proof is the equality
\begin{equation} \label{star} \;\: F_f^2 (t) =
2 \|f\|^2 - \langle f, T_t f\rangle - \overline{\langle f, T_t
f\rangle}, \end{equation} which follows by  a direct computation.

Given this equality the implication (vi) $\Longrightarrow$ (iii)'
follows immediately. Conversely, (ii)' implies almost periodicity of
$F_f$ for any real-valued $f\in C(X)$ and the above equality
\eqref{star} then gives almost periodicity of $\langle f, T_t
f\rangle $ for any such $f$. By a simple polarisation argument this
then gives almost periodicity of $\langle f, T_t g\rangle$ for all
real valued continuous $f,g$. This, in turn easily implies (v) from
which (v)' follows by denseness of $C(X)$ in $L^2 (X,m)$.

\smallskip

We now turn to the last statement: Set $S_f : G\longrightarrow \CC,
S_f (t) = \langle f, T_t f\rangle$ for $f\in L^2 (X,m)$.  Whenever
$f\in C(X)$ is real valued, we have $F_f^2  = 2 \|f\|^2 - 2 S_f$ by
\eqref{star}. Hence, we can lift $S_f $ to a continuous function on
$\TT (F_f)$. As  $F_f \prec \underline{d}$ due to Proposition
\ref{thecorollary} we also have $F_f'\prec d'$. This implies that
there exists a continuous map $\pi : \TT (\underline{d})
\longrightarrow \TT (F_f)$ mapping $\underline{d}(t +\cdot) $ to
$F_f (t+ \cdot)$ for any $t\in G$. Hence,  we can then  lift $S_f$
to a continuous function on $\TT (\underline{d})$ as well. By
considering real- and imaginary parts we can then lift $S_f$ to a
continuous function on $\TT (\underline{d})$ for any $f\in C(X)$.
Taking limits and using that $C(X)$ is dense in $L^2 (X,m)$ we can
then lift $S_f$ for any $f\in L^2 (X,m)$ to $\TT (\underline{d})$.
Hence, the set of frequencies of $S_f$ is contained in the set of
frequencies of $\underline{d}$ for any $f\in L^2 (X,m)$.

Conversely, consider $F:= F_{(f_n), (c_n)}$ for a set of functions
$f_n$ in $C(X)$  and $c_n > 0$, $n\in \NN$, with $\sum_n c_n \|f_n\|
<\infty$  such that the $(f_n)$ separate the points of $X$. Then,
$F\sim \underline{d}$. Hence, $\TT(F) = \TT (\underline{d})$. Thus,
the frequencies of $\underline{d}$ agree with the frequencies of
$F$, which in turn are contained in the group generated by the
frequencies associated to the $f_n$, $n\in\NN$.
\end{proof}

\section{A characterization of  discrete spectrum}\label{Characterization}
 In this section we state and prove
our main result which gives a characterization of dynamical systems
with  discrete spectrum.

\medskip

\begin{theorem}[Characterizing  discrete spectrum]
\label{mainresult} Let $(X,G,m)$ be a dynamical system. Let $d$ be a
metric on $X$ inducing the topology. Then, the following assertions
are equivalent:

\begin{itemize}

\item[(i)] The dynamical system $(X,G,m)$ has   discrete
spectrum.

\item[(ii)] The function $\underline{d}$ is almost periodic.
\end{itemize}
If one of these equivalent conditions hold the group of eigenvalues
of $(X,G,m)$ equals the group generated by the frequencies of
$\underline{d}$.

\end{theorem}

\textbf{Remark.} If  $\underline{d}$ is almost periodic,  the group
generated its frequencies is the dual group of $\TT (\underline{d})$
(as follows from the general discussion above). Moreover, it is not
hard to see that this group is also given as the Hausdorff
completion of $G$ with respect to $d'$.

\begin{proof}
We first deal with the equivalence statement. By the Lemma
\ref{characterization-ap}, the function $\underline{d}$ is almost
periodic if and only if $G\longrightarrow \CC, \;  t\mapsto \langle
f, T_t f\rangle=S_f (t) $ is almost periodic for any $f\in L^2
(X,m)$. This in turn is a well-known characterization of   discrete
spectrum. Indeed,  for any $f\in L^2 (X,m)$ there exists a unique
measure $\mu_f$  on $\widehat{G}$ with
$$\langle f, T_t f\rangle =
\int_{\widehat{G}} \gamma (t) d\mu_f (\gamma)$$ (see e.g.
\cite{Loomis}). Now, discrete spectrum just means that all these
measures are point measures and  pure pointedness of $\mu_f$ is
equivalent to almost periodicity of $t\mapsto \langle f, T_t
f\rangle$ (see e.g. \cite{LS} for recent discussion).

We now turn to the second statement of the theorem. By the  last
statement of Lemma \ref{characterization-ap} the group generated by
the frequencies of $\underline{d}$ is the group generated by the
frequencies of the $S_f$, $f\in L^2 (X,m)$. The frequencies of
$S_f$, however, are just the atoms of the measures $\mu_f$. Hence,
they generate  the group of eigenvalues.
\end{proof}

It is possible to rephrase the result in terms of almost-periods.
This will clarify the relationship between our result and earlier
results.  Whenever $e$ is a continuous  pseudometric on $X$ a $t\in
G$ is called a \textit{measure theoretic  $\varepsilon$-almost
period of $e$} if
$$m(\{x \in X : e(x,tx)>\varepsilon\}) <  \varepsilon.$$

\begin{lemma}[Almost periodicity and measure-theoretic $\varepsilon$
- almost periods]\label{theperiods} Let $e$ be a continuous
pseudometric on $X$. Then, the following assertions are equivalent:
\begin{itemize}
\item[(i)] For any $\varepsilon>0$ the set of measure theoretic $\varepsilon$-almost
periods of $e$ is relatively dense.
\item[(ii)] The function $\underline{e}$ is almost periodic.
\end{itemize}

\end{lemma}
\begin{proof}

(i)$\Longrightarrow$(ii): We have to show that the set of $\{ t\in G
: \underline{e} (t) \leq \varepsilon \}$ is relatively dense for any
$\varepsilon >0$ (compare Lemma \ref{observation}). Let
$\varepsilon_1>0$ with
$$\varepsilon_1  \|e\|_\infty + \varepsilon_1 < \varepsilon $$ be given.
Chose a measure theoretic  $\varepsilon_1$-almost period $t$  of $e$
and set
$$M:=\{x\in X : e(x,tx) >\varepsilon_1\}.$$
Then, a direct computation gives
$$\underline{e}(t) =  \int_M e(x,tx) dm +
\int_{X\setminus M} e(x,tx) dm \leq \varepsilon_1 \|e\|_\infty +
m(X\setminus M) \varepsilon_1 < \varepsilon.$$ As the set of measure
theoretic $\varepsilon_1$-almost periods is relatively dense by (i)
the desired statement follows.

\medskip

(ii)$\Longrightarrow$(i): This follows by mimicking an argument
given in the proof of Lemma \ref{thelemma}:  Let $\varepsilon >0$ be
given. By (ii) the set $\{t \in G: \underline{e}(t) <
\varepsilon^2\}$ is relatively dense. For any $t$ in this set we
obtain with $N:=\{x\in X : e(x,tx)
>\varepsilon\}$
$$\varepsilon m(N) \leq \int_X e(x,tx) dm = \underline{e} (t)  < \varepsilon^2$$
and, hence,  $m(N)< \varepsilon$. This finishes the proof.
\end{proof}

From the previous lemma and the main result, Theorem
\ref{mainresult}, we directly obtain the following consequence.

\begin{corollary}\label{cor-sol}
Let $(X,G,m)$ be a dynamical system and $d$ a metric on $X$ inducing
the topology on $X$. Then, the following assertions are equivalent:
\begin{itemize}
\item[(i)] For any $\varepsilon>0$ the set of measure theoretic
$\varepsilon$-almost periods of $d$  is relatively dense.
\item[(ii)] The dynamical system has  discrete  spectrum.
\end{itemize}
\end{corollary}

\textbf{Remarks.} The implication (i)$\Longrightarrow$(ii) of the
corollary is proven as Theorem 3.2 in \cite{Sol} (under an
additional ergodicity assumption).  There also  a partial converse
is provided in Proposition 3.3. This converse needs  an additional
requirement of continuity of eigenfunctions and it is remarked that
'it is unlikely' that a full converse of the theorem holds. Our
corollary shows that such a  full converse does hold.

\bigskip

We can also derive the following consequence.

\begin{corollary}\label{cor-eq}
Let $(X,G,m)$ be a dynamical system. Let $\mathcal{F}$ be a family
of continuous functions on $X$, which separates the points of $X$.
Then, the following assertions are equivalent:

\begin{itemize}
\item[(i)] For any $\varepsilon >0$ and each $f\in\mathcal{F}$ the
set of $\varepsilon$-almost periods of $\underline{e}_f$ is
relatively dense.

\item[(ii)] The dynamical system has   discrete  spectrum.

\end{itemize}

\end{corollary}
\begin{proof} (ii)$\Longrightarrow$(i): This follows from
combining the implication (i)$\Longrightarrow$(ii) of Theorem
\ref{mainresult}, the implication (i)$\Longrightarrow$(ii) of Lemma
\ref{characterization-ap} and the first lemma of this section. (This
reasoning actually works for any $f\in C(X)$ and not only for
$f\in\mathcal{F}$.)

\smallskip

(i)$\Longrightarrow$(ii): This follows from combining the first
lemma of this section with the implication (iii)$\Longrightarrow$(i)
of Lemma \ref{characterization-ap} and  Theorem \ref{mainresult}
(ii)$\Longrightarrow$(i).
\end{proof}

\textbf{Remark.} A possible choice of the family $\mathcal{F}$ is
given as $d(x,\cdot)$, $x\in X$.

\section{Connection to the autocorrelation
measure}\label{Connections} Our considerations are motivated by the
study of diffraction theory for  quasicrystals. Diffraction theory
for quasicrystals and the relationship with dynamical systems has
gained substantial  attention in the last two  decades. Indeed, from
the very beginning tiling and point set dynamical  systems with
discrete spectrum have played a key role
\cite{Dwo,Rob,Sol,Sol2,Schl}.  We refer the interested reader to the
surveys \cite{BL-survey,Len-survey} and the corresponding parts of
\cite{Lenz} for further details and references. In this section we
briefly sketch the necessary background to put our main result into
this  context.

\medskip

As discussed in \cite{BL}, diffraction theory for quasicrystals can
conveniently be set up in the framework of translation bounded
measures on a locally compact, $\sigma$-compact  abelian group $G$.
In fact, one can  go beyond translation bounded measures and deal
with measures satisfying only existence of a suitable second moment
\cite{LS}  or even deal with suitable pairings of functions on $G$
and the elements of the dynamical system \cite{LM}. However, here we
will stick to translation bounded measures as this seems to
encompass the usual models for quasicrystals. We follow \cite{BL} to
which we refer for further details and references.

Let $C_c (G)$ be the space of continuous functions on $G$ with
compact support. A measure $\mu$ on $G$ is called
\textit{translation bounded} if its total variation $|\mu|$
satisfies
$$\sup |\mu| (t + U) < \infty$$
for one (all) relatively compact open $U$ in $G$. The set of all
translation bounded measures is denoted by $M^\infty (G)$. It is
equipped with the vague topology. There is a natural action $\alpha$
of $G$ on $M^\infty (G)$ by translations where for $t\in G$ and
$\mu\in M^\infty (G)$ the measure  $\alpha_t (\mu)$ satisfies
$\alpha_t (\mu) (\varphi) = \mu (\varphi (\cdot - t))$ for all
$\varphi \in C_c (G)$. Whenever $X$ is a compact  subset  of
$M^\infty (G)$ which is invariant under the translation action and
$m$ is an invariant probability measure on $X$, we call $(X,G,m)$ a
\textit{dynamical system of translation bounded measures} or just
TMDS for short. Such a system comes with a canonical  map
$$N : C_c (G)\longrightarrow C (X) \mbox{ with } N_\varphi (\mu) =
\int \varphi (-s) d\mu (s).$$ Let us emphasize that the existence of
such a map is a distinctive feature of TMDS compared to general
dynamical systems. There exists then a unique translation bounded
measure $\gamma$ on $(X,G,m)$ with

\begin{equation}\label{gamma}
\gamma \ast \varphi \ast \widetilde{\varphi} (t) = \langle
N_\varphi, T_t N_\varphi\rangle
\end{equation} for all $\varphi \in
C_c (G)$ and all $t\in G$.  Here, $\ast$ denotes the convolution and
$\widetilde{\varphi} (t) = \overline{\varphi} (-t)$.  The measure
$\gamma$ is called the \textit{autocorrelation} of the TMDS. This
measure allows for a Fourier transform $\widehat{\gamma}$ which is a
(positive) measure on $\widehat{G}$.  A main result of the theory
gives the following:

\medskip

\textbf{Theorem.} \textit{ The TMDS $(X,G,m)$ has discrete spectrum
if and only if the measure $\gamma$ is strongly almost periodic. In
this case, the group of eigenvalues of $(X,G,m)$ is the group
generated by  $\{ k\in\widehat{G} : \widehat{\gamma} (\{k\})
>0\}$.}

\bigskip

A few \textbf{remarks} are in order:

\begin{itemize}

\item The measure $\gamma$ is called strongly almost periodic of
$\gamma \ast \varphi$ is almost periodic (in the sense of Bohr) for
all $\varphi \in C_c (G)$.

\item The above
theorem is usually formulated with the assumption that
$\widehat{\gamma}$ is a pure point measure. However, as is
 well-known, see e.g. Proposition 7 and Theorem 4  in \cite{BM} for
a discussion in the context of aperiodic order, the measure $\gamma$
is strongly almost periodic if and only if $\widehat{\gamma}$ is a
pure  point measure.

\item The theorem has a long history. The
connection between diffraction measure and  point spectrum goes back
to work of Dworkin \cite{Dwo}. The first statement giving an
equivalence (in the more restricted setting of uniquely ergodic
dynamical systems of point sets satisfying the regularity
requirement of finite local complexity) can be found in \cite{LMS}.
This was then generalized (in slightly different directions) in
\cite{Gou} and \cite{BL}. A unified treatment of \cite{BL,Gou} was
given in \cite{LS}. Recently an even more general result was given
in \cite{LM}. In the context of TMDS discussed in this section  the
theorem was first proven in \cite{BL}.

\end{itemize}

Clearly, the preceding result is quite close to our main result:
Pure pointedness of the spectrum is characterized by almost
periodicity of a suitable function  (in this case the measure
$\gamma$) and the group of eigenvalues is generated by the
frequencies of the function (in this case the atoms of the
diffraction measure). In fact, we can derive the previous result
from our main result (provided the group is metrizable) along the
following lines:

The measure $\gamma$ is strongly almost periodic if and only if
$\gamma \ast \varphi \ast \widetilde{\varphi}$ is almost periodic
for all   $\varphi \in C_c (G)$. As $G$ is metrizable and
$\sigma$-compact, there exists  a dense set $\varphi_n$, $n\in\NN$,
of functions in  $C_c (G)$ and $\gamma$ is almost periodic if and
only if $\gamma\ast \varphi_n \ast \widetilde{\varphi_n}$ is almost
periodic for each $n\in\NN$.

By the formula given above in \eqref{gamma} this is equivalent to
$t\to \langle N_{\varphi_n}, T_t N_{\varphi_n}\rangle$ being almost
periodic for all  $n\in\NN$. As the $\varphi_n$, $n\in\NN$, are
dense in $C_c (G)$, they clearly separate the points of $X$. Hence,
condition (vi) of Lemma \ref{characterization-ap} is satisfied. The
lemma then gives almost periodicity of $\underline{d}$ and Theorem
\ref{mainresult} shows discrete spectrum. The statement on the
frequencies follows by going through the argument and keeping track
of the involved frequencies.

\medskip

\textbf{Remark.} In the case of TMDS  the map $N$ and the Fourier
transform $\widehat{\gamma}$ of the autocorrelation measure $\gamma$
arise by a limiting procedure out of quantities which are defined
for any dynamical system. Details are discussed in this remark:
Whenever $(X,G,m)$ is a dynamical system each $f\in L^2 (X,m)$ gives
rise to the  map
$$N^f : C_c (G)\longrightarrow L^2(X,m)$$
via $N^f (\varphi):= \int \varphi (-t) T_t f d t$. A direct
computation then shows for any $\varphi, \psi \in C_c (G)$
\begin{eqnarray}\label{drei}\langle N^f (\varphi), N^f (\psi)\rangle  = \int_{\widehat{G}}
\overline{\widehat{\varphi}} \;  \widehat{\psi} \;  d\mu_{f}
(\gamma),
\end{eqnarray} where $\widehat{\sigma}$ denotes the Fourier transform of
$\sigma$ defined by
$$\widehat{\sigma} (\gamma) = \int
\overline{\gamma (t)} \sigma (t) d m_G(t)$$ and $\mu_f$ is the
spectral measure of $f$ (i.e. the  unique finite measure  $\mu_f$ on
$\widehat{G}$ with $\langle f, T_{t} f\rangle = \int_{\widehat{G}}
\gamma (t) d\mu_f (\gamma)$ for all $t\in G$). In this way, we
obtain for any specific $f$ a map $N^f$, which encodes the spectral
measure of $f$ in the sense that \eqref{drei} holds. It is not hard
to conclude from \eqref{drei} that there exists a unique isometry
$$\Theta^f : L^2 (\widehat{G},\mu_f)\longrightarrow L^2 (X,m)$$
with $\Theta^f (\widehat{\sigma}) = N^f (\sigma)$ for any $\sigma
\in C_c (G)$. Indeed, this can be seen as one version of the
spectral theorem.  If $(X,G,m)$ is a TMDS, we can  replace  $f$ by a
canonical limiting object. More specifically, we can consider an
approximate unit $(\varphi_\alpha)$ in $C_c (G)$ and define
$f_\alpha \in L^2 (X,m)$ by $f_\alpha (\mu) := \int
\varphi_\alpha(-s) d\mu (s)$. Then,
$$N^{f_\alpha} (\varphi) \to N (\varphi) \mbox{ in $L^2 (X,m)$}. $$
Indeed, by construction and the defining properties of an
approximate unit  we have pointwise (in $\mu \in X$) convergence of
$N^{f_\alpha}(\varphi) (\mu) = \varphi \ast (\varphi_\alpha \ast\mu)
(0)$ to $ \varphi \ast \mu (0) = N (\varphi) (\mu)$ as well as a
uniform (in $\mu \in X$) bound $|N^{f_\alpha} (\varphi) (\mu)|\leq
\|f\|_\infty |\mu| (K) \leq C <\infty$ due to the assumption on $X$.
Thus, the maps $N^{f_\alpha}$ converge to the map $N$. Moreover, the
spectral measures $\mu_{f_\alpha}$ converge to $\widehat{\gamma}$
(see Corollary 1 in \cite{BL}). So, both $N$ and $\widehat{\gamma}$
arise by a limiting procedure which involves an approximate unit.

\section*{Acknowledgments} The author gratefully acknowledges
recent inspiring discussions with Gabriel Fuhrmann, Maik Gr\"oger
und Tobias J\"ager. The present work also owes to a highly
enlightening discussion on almost periodicity  with Jean-Baptiste
Gou\'{e}r\'{e} in 2004 and to various discussions with Michael
Baake, Robert V. Moody and  Nicolae Strungaru.

\end{document}